\documentclass{article}

\usepackage[all]{xy}
\usepackage{mathrsfs}
\usepackage{amssymb}
\usepackage{amsthm}
\usepackage{amsmath}

\usepackage{tikz}
\usetikzlibrary{arrows,positioning}

\newtheorem{thm}{Theorem}

\newtheorem{prop}[thm]{Proposition}
\newtheorem{cor}[thm]{Corollary}

\newtheorem{defn}[thm]{Definition}

\newtheorem{rem}[thm]{Remark}

\author{Camilo Sanabria Malag\'on\\Department of Mathematics\\Universidad de los Andes, Bogot\'a, Colombia}

\title{Solutions of algebraic linear ordinary differential equations}

\begin{document}

\date{}

\maketitle

\begin{abstract}
A classical result of F.Klein states that, given a finite primitive group $G\subseteq SL_2(\mathbb{C})$, there exists a hypergeometric equation such that any second order LODE whose differential Galois group is isomorphic to $G$ is projectively equivalent to the pullback by a rational map of this hypergeometric equation. In this paper, we generalize this result. We show that, given a finite primitive group $G\subseteq SL_n(\mathbb{C})$, there exist a positive integer $d=d(G)$ and a standard equation such that any LODE whose differential Galois group is isomorphic to $G$ is gauge equivalent, over a field extension $F$ of degree $d$, to an equation projectively equivalent to the pullback by a map in $F$ of this standard equation. For $n=3$, these standard equations can be chosen to be hypergeometric.
\end{abstract}

\section*{Introduction}

In the study of linear ordinary differential equations (LODEs), one of the central questions is to find closed form solutions. Here, we give a solution to this question that can be algorithmically implemented for algebraic LODEs, i.e. LODEs whose full system of solutions can be expressed in terms of algebraic functions. We base our approach on the Picard-Vessiot theory which is a galoisian perspective developed by E. Picard \cite{PICARD1991III}, E. Vessiot \cite{VESSIOT1892} and E. R. Kolchin \cite{KOLCHIN1946}.

In the literature there are three different methods for describing the solutions of algebraic LODEs.

In \cite{SINGER1993}, M. Singer and F. Ulmer showed how to compute the minimal polynomial of the solutions of algebraic LODEs. The drawback of this method is that the degree of the minimal polynomial is relatively big, it can reach $120$ for second order equations and $216$ for third order.

A second approach has been given by J. Kovacic in \cite{KOVACIC1986}. J. Kovacic showed that the solutions of second order algebraic LODEs can be expressed as exponential of an integral of an algebraic function, where the degree of the minimal polynomial of the algebraic function is at most $12$. M. van Hoeij, J. F. Ragot, M. Singer, F. Ulmer and J.-A. Weil \cite{VANHOEIJ1999} extended this method to higher orders and showed that, for equations of third order, the minimal polynomial of the algebraic function involved in the integral of the exponential is of degree up to $45$.

A third and the most effective method for expressing the solutions of algebraic LODEs is based on a classical result of F. Klein \cite{KLEIN187711,KLEIN187712}. This result states that a full system of solutions to an algebraic second order LODE with rational coefficients can be given in closed form as solutions to hypergeometric equations precomposed with a rational function, multiplied by the solution to a first order LODE. Furthermore, for each finite group $G\subseteq SL_2(\mathbb{C})$, there exist a standard hypergeometric equation which is minimal in the sense that any other LODE with differential Galois group $G$ is projectively equivalent to a pullback by a rational map of this hypergeometric equation. B. Dwork and F. Baldassari \cite{BALDASSARRI1980,BALDASSARRI1979} presented Klein's result using modern tools and generalized it to second order algebraic LODEs defined over algebraic curves. Moreover, J.-A. Weil and M. van Hoeij \cite{VANHOEIJ2005} identified a standard equation for each of the primitive finite group in $SL_2(\mathbb{C})$ and implemented Klein's result in an algorithm. By introducing the concept of standard equation, M. Berkenbosch \cite{BERKENBOSCH2006} generalized Klein's result to order three. C. Sanabria  \cite{SANABRIA2017} extended Klein-Dwork-Baldassari's result to any order showing that any algebraic LODE is projectively equivalent to the pullback by a rational map of a standard equation. A drawback of this generalization is that the family of standard equations involved is infinite and continuous, making it unsuitable for an algorithmic implementation.

In this paper, we improve the result of C. Sanabria \cite{SANABRIA2017} by proving that for each finite primitive group $G\subseteq SL_n(\mathbb{C})$, there exist a positive integer $d$ and a standard equation such that the solutions to any LODE with differential Galois group isomorphic to $G$ can be expressed in terms of the solutions of this standard equation, together with its derivatives, functions in a field extension of degree $d$, and a solution to a first order LODE defined over this extension. With this result, we reduce the infinite family of standard equations to a singleton. In the last section, we show how to obtain closed form solutions to irreducible algebraic LODEs of order three in terms of generalized hypergeometric functions. The computations were carried out using MAPLE and can be obtained from my webpage \cite{MapleComputations2021}.

\section{Preliminary}

\subsection{Space of orbits}\label{orbitspace}

Let $G\subseteq GL_n(\mathbb{C})$ be a finite linear algebraic group. We define a left group action of $G$ on the coordinate ring $R=\mathbb{C}[X_1,\ldots,X_n]$ of $\mathbb{C}^n$ by $\mathbb{C}$-automorphisms
\[
X_j\mapsto \sum_{i=1}^n X_ig_{ij},
\]
for $(g_{ij})_{i,j=1}^n\in G$.
This $G$-action on $R$ induces a right $G$-action on $\mathbb{C}^n$ defined by
\[
(x_1,\ldots,x_n)\mapsto \Big(\sum_{i=1}^nx_ig_{i1},\ldots,\sum_{i=1}^nx_ig_{in}\Big).
\]
Since $G$ is finite, it is reductive and thus the $G$-invariant polynomials in $R$ separate the $G$-orbits in $\mathbb{C}^n$. Therefore, the coordinate ring of the orbit space $\mathbb{C}^n/G$ is the finitely generated $G$-invariant subring $R^G$.

Let $F_1,\ldots,F_N$ be homogeneous generators of $R^G$. We can embed $\mathbb{C}^n/G$ into $\mathbb{C}^N$ through the algebraic map
\[
\mathbf{x}\cdot G \mapsto \Big(F_1(\mathbf{x}),\ldots,F_N(\mathbf{x})\Big),
\]
where $\mathbf{x}=(x_1,\ldots,x_n)$. We will identify the orbit space $\mathbb{C}^n/G$ with its embedding in $\mathbb{C}^N$. Since $G$ is finite, the orbit space has dimension $n$ and therefore, there exits a dense subset of $\mathbb{C}^n$ where the derivative of the quotient map
\[
\mathbf{x}\mapsto \Big(F_1(\mathbf{x}),\ldots,F_N(\mathbf{x})\Big)
\]
is non-singular.

Similarly, these $G$-actions define $G$-actions on $\mathbb{P}^{n-1}(\mathbb{C})$ and on its homogeneous coordinate ring $R$. The orbit space of the action on $\mathbb{P}^{n-1}(\mathbb{C})$ is $\mathbb{P}^{n-1}(\mathbb{C})/G$. Let $\Lambda\in\mathbb{Z}_{>0}$ be such that the homogeneous elements of $R^G$ of degree $\Lambda$, $R^G_\Lambda$, form a homogeneous coordinate system for $\mathbb{P}^{n-1}(\mathbb{C})/G$, i.e. $$\mathbb{P}(R^G_\Lambda)\simeq\mathbb{P}^{n-1}(\mathbb{C})/G.$$
Let $\Phi_1,\ldots,\Phi_M$ be a basis of $R^G_\Lambda$. We can embed $\mathbb{P}^{n-1}(\mathbb{C})/G$ into $\mathbb{P}^{M-1}(\mathbb{C})$ with the map
\[
[\mathbf{x}]\cdot G\mapsto \Big[\Phi_1(\mathbf{x}):\ldots:\Phi_M(\mathbf{x})\Big],
\]
where $[\mathbf{x}]=[x_1:\ldots:x_n]$. As before, we identify the orbit space $\mathbb{P}^{n-1}(\mathbb{C})/G$ with its embedding in $\mathbb{P}^{M-1}(\mathbb{C})$. Again, there exits a dense subset of $\mathbb{P}^{n-1}(\mathbb{C})$ where the derivative of the quotient map
\[
[\mathbf{x}]\mapsto \Big[\Phi_1(\mathbf{x}):\ldots:\Phi_M(\mathbf{x})\Big]
\]
is non-singular.

\subsection{Schwarz maps}

Let $C_0$ be a compact Riemann surface and let $K=\mathbb{C}(C_0)$ be the field of meromorphic functions over $C_0$. Let $\delta:K\rightarrow K$ be any non-trivial derivation. Note that $\delta$ can be uniquely extended to the sheaf of meromorphic functions over any open set in $C_0$.

Let $L(y)=0$ be an LODE of order $n$ with rational coefficients
\[
L(y)=\delta^n(y)+a_{n-1}\delta^{n-1}(y)+\ldots+a_1\delta(y)+a_0y,
\]
where $a_0,a_1,\ldots,a_{n-1}\in K$. Let $S\subset C_0$ be the collection of singularities of $L(y)$. Given a non-singular point $p\in C_0$ together with a fundamental system of solutions $\mathbf{y}=(y_1,\ldots,y_n)$ over a neighborhood $U\subseteq C_0$ of $p$, we define the Schwarz map as the analytic extension of 
\begin{eqnarray*}
[\mathbf{y}]: U & \longrightarrow & \mathbb{P}^{n-1}(\mathbb{C})\\
z & \longmapsto & \Big[y_1(z):\ldots:y_n(z)\Big].
\end{eqnarray*}

The monodromy of the Schwarz map is the projection of the monodromy group of the fundamental system of solutions $(y_1,\ldots,y_n)$ on $PGL_n(\mathbb{C})$. We will denote it by $G_0$ and call it the projective monodromy of $L(y)=0$. Let us assume $G_0$ is finite. Then, post-composing the Schwarz map with the quotient by the action of $G_0$ on  $\mathbb{P}^{n-1}(\mathbb{C})$, we obtain a single-valued map $$\psi:C_0\setminus S\rightarrow\mathbb{P}^{n-1}(\mathbb{C})/G_0,$$ which we will call the quotient Schwarz map.

\subsection{Algebraic Picard-Vessiot theory}

We will briefly recall the results from Picard-Vessiot theory that we will use in this paper. An extensive exposition of this subject is given in \cite{VANDERPUT2003}.

We will assume that $L(y)=0$ is irreducible and that all its solutions are algebraic over $K$. Under these assumptions, if $\mathbf{y}=(y_1,\ldots,y_n)$ is a fundamental system of solutions, then a Picard-Vessiot extension for $L(y)=0$ is the field
$$E=K[y_1,\ldots,y_n].$$
Let us denote by $I$ the kernel of the $K$-morphism
\begin{eqnarray*}
\Phi: K[Y_1,\ldots,Y_n] & \longrightarrow & K\\
Y_j & \longmapsto & y_j.
\end{eqnarray*}
An isomorphic representation of the differential Galois group of $L(y)=0$ is the finite group $G\subset GL_n(\mathbb{C})$ composed by the elements $(g_{ij})_{i,j=1}^n$ that send the ideal $I$ into itself under the $K$-automorphism
\begin{eqnarray*}
K[Y_1,\ldots,Y_n]  & \longrightarrow & K[Y_1,\ldots,Y_n]\\
Y_j & \longmapsto & \sum_{i=1}^nY_ig_{ij}.
\end{eqnarray*}
We will call $G$ the representation induced by $\mathbf{y}$. By the Galois correspondence, if $P\in K[Y_1,\ldots,Y_n]$ is $G$-invariant, then $\Phi(P)\in K$. Since $L(x)=0$ is irreducible, then $G$ is reductive and we have the following theorem.

\begin{thm}\label{compointthm}[Compoint's theorem for the algebraic case \cite{SANABRIA2017,VANDERPUT2020}] Let $L(y)=0$ be an algebraic irreducible LODE. If the differential Galois group of $L(y)=0$ is finite, then the ideal $I$ is generated by the $G$-invariants contained in it. In particular, if $P_1,\ldots,P_N\in\mathbb{C}[Y_1,\ldots,Y_n]$ is a set of generators of the $G$-invariant subring $\mathbb{C}[Y_1,\ldots,Y_n]^G$, then
\begin{equation*}\label{Igen}
I=\langle P_1-f_1,\ldots, P_N-f_N\rangle,
\end{equation*}
where $f_i=\Phi(P_i)$, $i=1,\ldots,N$.
\end{thm}

\begin{rem}
Given $L(y)=0$, as in the previous theorem, and homogeneous generators $P_1,\ldots, P_N$ of $\mathbb{C}[Y_1,\ldots,Y_n]^G$, the values $f_1,\ldots,f_n$ can obtained using the algorithm in \cite{VANHOEIJ1997}. 
\end{rem}

\subsection{Gauge equivalence and invariants}

\begin{defn}
Let $L(y)=0$ and $L'(x)=0$ be LODEs of order $n$ with coefficients in $K$, let $K\subseteq K_1$ be a field extension, and let $\delta:K\rightarrow K$ be a non-trivial derivation. We say that $L'(x)=0$ is gauge equivalent to $L(y)=0$ over $K_1$ if there exist $f_0,f_1,\ldots,f_{n-1}\in K_1$ such that $x=f_0y+f_1\delta(y)+\ldots+f_{n-1}\delta^{n-1}(y)$ is a solution of $L'(x)=0$ whenever $y$ is a solution of $L(y)=0$.
\end{defn}

Let $V$ be a  $K$-vector space and let $\delta:K\rightarrow K$ be a non-trivial derivation. A connection $\nabla: V\rightarrow V$ over $(K,\delta)$ is an additive map that satisfies the Leibniz rule $\nabla fv=\delta(f)v+f\nabla v$. A differential module over $(K,\delta)$ is a $K$-vector space endowed with a connection. If $TV$ is a tensorial construction of $V$, then $TV$ inherits a connection $T\nabla$ and therefore, a differential module structure.

The LODE $L(y)=0$ defines the connection $\nabla_L: K^n \rightarrow K^n$ such that
$$\nabla_L(e_i) =  e_{i+1},\quad i=1,\ldots,n-1$$
and
$$\nabla_L(e_n) = - a_{n-1}e_n-\ldots-a_1e_2-a_0e_1,$$
for the canonical basis $(e_1,\ldots,e_n)$.
There is a well know correspondence between the homogeneous invariants of degree $d$ in $K[Y_1,\ldots,Y_n]^G_d$ and the solutions to $S^d\nabla_L v=0$, where $S^d$ is the $d$-th symmetric power. 

Let $K\subseteq K_1$ be a field extension. Let $f_0,f_1,\ldots,f_{n-1}\in K_1$ and let
$$L'(x)=\delta^n(x)+b_{n-1}\delta^{n-1}(x)+\ldots+b_1\delta(x)+a_0x=0$$
be the LODE such that $x=f_0y+f_1\delta(y)+\ldots+f_{n-1}\delta^{n-1}(y)$ is a solution of $L'(x)=0$ whenever $y$ is a solution of $L(y)=0$. Then, if $(e_1,\ldots,e_n)$ is the canonical basis of $K_1^n$, $\nabla_L: K_1^n \rightarrow K_1^n$ is the connection defined by $L(y)=0$, and if we define
$$e'_1 = f_0e_1+f_1e_2+\ldots+f_{n-1}e_n$$
and
$$e'_i=\nabla_Le'_{i-1},\quad i=2,\ldots,n,$$
then
$$\nabla_L(e'_n)=- b_{n-1}e'_n-\ldots-b_1e'_2-b_0e'_1.$$
In particular, if $x_i=f_0y_i+f_1\delta(y_i)+\ldots+f_{n-1}\delta^{n-1}(y_i)$, $i=1,\ldots,n$, then the correspondence between homogeneous invariants and solutions to the symmetric power \cite{SANABRIA2014} implies that for every $P\in K[Y_1,\ldots,Y_n]^G_d$, $P(x_1,\ldots,x_n)$ can be written as a linear combination of elements in $\mathbb{C}[y_1,\ldots,y_n]^G_d$ with coefficients in $K_1$. We obtain the following result.

\begin{prop}\label{propgau}
Let $L(y)=0$ and $L'(x)=0$ be LODEs of order $n$ with coefficients in $K$ that are gauge equivalent over $K_1\supset K$. Let $\delta:K\rightarrow K$ be a non-trivial derivation. Then we have the following.
\begin{itemize}
\item[i)] The differential Galois groups of $L(y)=0$ and $L'(x)=0$ over $K_1$ are isomorphic.
\item[ii)] If $\mathbf{y}=(y_1,\ldots,y_n)$ and $\mathbf{x}=(x_1,\ldots,x_n)$ are fundamental systems of solutions of $L(y)=0$ and $L'(x)=0$, respectively, such that they induce the same representation $G\subset GL_n(\mathbb{C})$ of their differential Galois groups, then, given $P\in K[Y_1,\ldots,Y_n]^G_d$ and a basis $Q_1,\ldots,Q_m$ of $\mathbb{C}[Y_1,\ldots,Y_n]^G_d$, there exist $g_1,\ldots,g_m\in K_1$ such that
$$P(\mathbf{x})=g_1Q_1(\mathbf{y})+\ldots+g_mQ_m(\mathbf{y}).$$
\item[iii)] If $f_0,f_1,\ldots,f_{n-1}\in K_1$ are such that
$$x_i=f_0y_i+f_1\delta(y_i)+\ldots+f_{n-1}\delta^{n-1}(y_i),\ i=1,\ldots,n,$$
then $g_i\in \mathbb{C}[f_0,\ldots,f_{n-1}]^G_d$.
\end{itemize}  
\end{prop}

\subsection{Schwarz maps for invariant curves}

Let $G$ be a finite primitive algebraic subgroup of $GL_n(\mathbb{C})$ and let $C\subseteq\mathbb{P}^{n-1}(\mathbb{C})$ be an algebraic $G$-invariant curve not contained in a projective hyperplane. Then, by the following theorem \cite{SANABRIA2017arxiv}, there exist an LODE $L_C(y)=0$ such that the image of its Schwarz is $C$. Furthermore, the field of definition of this equation is an abelian extension of the field of meromorphic functions of $C$.

\begin{thm}\label{theoalg}
Let $G$ be a finite primitive algebraic subgroup of $GL_n(\mathbb{C})$ and let $C\subseteq\mathbb{P}^{n-1}(\mathbb{C})$ be an algebraic $G$-invariant curve not contained in a projective hyperplane. Let $C_0$ be a compact Riemann surface  such that there is a non-constant morphism of curves
$$\psi: C_0\setminus S\rightarrow C/G\subseteq\mathbb{P}^{n-1}(\mathbb{C})/G,$$
with $S\subseteq C_0$ finite, and such that $C\rightarrow C/G$ is unramified over $\psi(C_0\setminus S)$ (i.e. $C\rightarrow C/G$ is a local isomorphism of complex curves over every point  of $\psi(C_0\setminus S)$). Let $K_0=\mathbb{C}(C_0)$ be the field of meromorphic functions over $C_0$. Then there exists a branched cover $\pi:C_1\rightarrow C_0$, where $C_1$ is a compact Riemann surface and $K_0\subseteq K_1=\mathbb{C}(C_1)$ is an abelian extension, and there exists a unique monic linear differential equation $L_C(y)=0$ of order $n$ with coefficients in $K_1$ admitting a fundamental system of solutions $(y_1,\ldots,y_n)$ over an open set $U\subseteq C_1$, such that the closure of the image of its Schwarz map is $C$ (see diagram).
Moreover, the restriction of $\psi\circ\pi$ to $C_1\setminus\pi^{-1}(S)$ is the quotient Schwarz map of $L(y)=0$ associated to the system of solutions $(y_1,\ldots,y_n)$. 
\end{thm}

{\center
\hfill
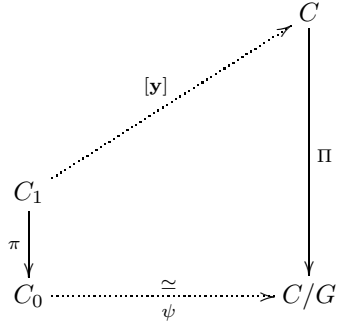
\xymatrix{
 & & & C\ar[ddd]^\Pi\\
 & & & \\
C_1\ar@{..>}[rrruu]^{[\mathbf{y}]}\ar[d]_\pi &  &  & \\
C_0\ar@{..>}[rrr]_\psi^\simeq & & &  C/G
}
\hfill
}

\section{Invariant curves, pullbacks and gauge equivalence}

\subsection{Invariant curves and pullbacks}

Let $G$ be a finite primitive algebraic subgroup of $GL_n(\mathbb{C})$ and let $C\subseteq\mathbb{P}^{n-1}(\mathbb{C})$ be an algebraic $G$-invariant curve not contained in a projective hyperplane. Let $L_C(y)=0$ be as in Theorem \ref{theoalg}. Suppose now that we have another $n$-th order LODE $L(x)=0$ admitting a fundamental system of solutions $(x_1,\ldots,x_n)$ such that the image of the associated Schwarz map is $C$. We will prove that $L(x)=0$ is projectively equivalent to the pullback of $L_C(y)=0$ by a map in an abelian extension of the field of definition $L$.

\begin{thm}\label{theopullback}
Let $G$ be a finite primitive algebraic subgroup of $GL_n(\mathbb{C})$ and let $C\subseteq\mathbb{P}^{n-1}(\mathbb{C})$ be an algebraic $G$-invariant curve not contained in a projective hyperplane. Let $C_0$, $C_1$, and $L_C(y)=0$ be as in Theorem \ref{theoalg}. Let $X_0$ be a compact Riemann surface, let $F_0=\mathbb{C}(X_0)$ be the field of meromorphic functions over $X_0$, and let $L(x)=0$ be an LODE with coefficients in $F_0$ such that, for a fundamental system of solutions, the closure of the image of its associated Schwarz map is $C$. Then there exists a branched cover $p:X_1\rightarrow X_0$, where $X_1$ is a compact Riemann surface and $F_0\subseteq F_1=\mathbb{C}(X_1)$ is an abelian extension, and there exist two functions $f,h\in F_1$ and $\Lambda\in\mathbb{Z}_{>0}$ such that $(x_1,\ldots,x_n)$, with $$x_i=f^{1/\Lambda}\cdot y_i\circ h,\quad i=1,\ldots,n,$$ is a fundamental system of solutions of $L(x)=0$ over an open set $V\subseteq C_1$.
\end{thm}

\begin{proof}
Let $\mathbf{z}=(z_1,\ldots,z_n)$ be a fundamental system of solutions of $L(x)=0$ over an open set $U_0\subseteq X_0$ such that the image of its Schwarz map is $C$. We may assume that the domain $U$ of $\mathbf{y}=(y_1:\ldots:y_n)$ and $U_0$ are such that the image of $[\mathbf{z}]=[z_1:\ldots:z_n]$ is included in the image of $[\mathbf{y}]=[y_1:\ldots:y_n]$. Let $\Pi: \mathbb{P}^{(n-1)}(\mathbb{C})\rightarrow \mathbb{P}^{(n-1)}(\mathbb{C})/G$ be the projection onto the orbit space so that $\Pi\circ[z_1:\ldots:z_n]$ maps $U_0$ into $C/G$.

Let $\psi$ and $\pi$ be as in Theorem \ref{theoalg} and let $\phi$ be an inverse of $\psi$. Then $\phi\circ\Pi\circ[\mathbf{z}]$ defines an algebraic map $h_0:X_0\rightarrow C_0$. We will identify $K_0=\mathbb{C}(C_0)$ with $h_0^*(K_0)\subseteq F_0$. Let $F_1=F_0[K_1]$, where $K_1=\mathbb{C}(C_1)$. Then $F_0\subseteq F_1$ is an abelian extension. Let $X_1$ be the Riemann surface such that $\mathbb{C}(X_1)=F_1$ and let $p:X_1\rightarrow X_0$ and $h:X_1\rightarrow C_1$ be the morphisms induced, respectively, by the inclusions $F_0\subseteq F_1$ and $K_1\subseteq F_1$. We may assume that $U_0=p(V)$ and $U=h(V)$, for some open set $V\subseteq C_1$.

By definition, we have $h_0\circ p=\pi \circ h$. Furthermore, $h_0=\phi\circ\Pi\circ[\mathbf{z}]$ over $U_0$. On the other hand, we have $\Pi\circ [\mathbf{y}] = \psi\circ\pi$ over $U$. Therefore,
\begin{align*}
\Pi\circ  [\mathbf{y}]\circ h &=\psi\circ\pi\circ h=\psi\circ h_0\circ p\\
 &=\psi\circ \phi \circ\Pi\circ[\mathbf{z}]\circ p\\
 & =\Pi\circ[\mathbf{z}]\circ p
\end{align*}
over $V$. Then there exist $(g_{ij})_{i,j=1}^n\in G$ such that $(g_{ij}):  [\mathbf{y}]\circ h\mapsto [\mathbf{z}]\circ p$. Since the image of $[\mathbf{z}]$ is included in the image of $[\mathbf{y}]$, we have that $(g_{ij})$ is the identity and $[\mathbf{y}]\circ h=[z]\circ p$.

As in Section \ref{orbitspace}, let $\Lambda\in\mathbb{Z}_{>0}$ be such that the homogeneous elements of $\mathbb{C}[X_1,\ldots,X_n]^G$ of degree $\Lambda$ form a homogeneous coordinate system for $\mathbb{P}^{n-1}(\mathbb{C})/G$ and let $\Phi_1,\ldots,\Phi_M$ be a basis of $\mathbb{C}[X_1,\ldots,X_n]^G_\Lambda$. Then, since $\Pi\circ  [\mathbf{y}]\circ h=\Pi\circ[\mathbf{z}]\circ p$, we have
\[
\Big[\Phi_1(\mathbf{z}\circ p):\ldots:\Phi_M(\mathbf{z}\circ p)\Big]=\Big[\Phi_1(\mathbf{y}\circ h):\ldots:\Phi_M(\mathbf{y}\circ h)\Big].
\]
Hence, there exists $f\in F_1$ such that $\Phi_i(\mathbf{z}\circ p)=f\cdot\Phi_i(\mathbf{y}\circ h)$, $i=1\ldots,M$. Since $\Pi$ is locally a biholomorphism, we have $\mathbf{z}\circ p=f^{1/\Lambda}\cdot\mathbf{y}\circ h$. The theorem follows by taking $x_i=z_i\circ p$, $i=1,\ldots,n$.

{\center
\hfill
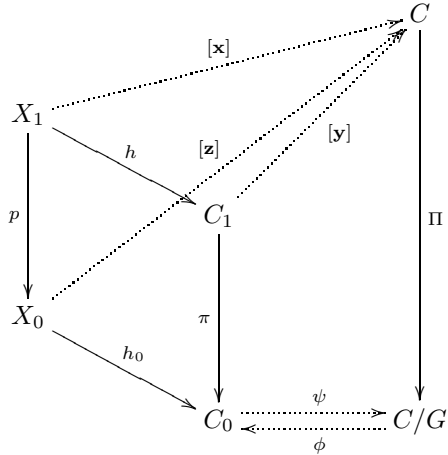
\xymatrix{
 & & & & C\ar[dddd]^\Pi\\
X_1\ar@{..>}[rrrru]^{[\mathbf{x}]}\ar[dd]_p\ar[rrd]^h & & & & \\
 & & C_1\ar@{..>}[rruu]_{[\mathbf{y}]}\ar[dd]_\pi &  & \\
X_0\ar@{..>}[rrrruuu]^{[\mathbf{z}]}\ar[rrd]^{h_0} & & & &\\ 
& & C_0\ar@{..>}@<0.7ex>[rr]^\psi & &  C/G\ar@{..>}@<0.7ex>[ll]^{\phi}
}
\hfill
}

\end{proof}

\subsection{Invariant curves and gauge equivalence}

Let $G$ be a finite primitive algebraic subgroup of $GL_n(\mathbb{C})$ and let $C$ and $X$ be algebraic $G$-invariant curves in $\mathbb{P}^{n-1}(\mathbb{C})$, none contained in a projective hyperplane. Let $L_C(y)=0$ and $L_{X}(x)=0$ be LODEs such that for two fundamental systems of solutions $\mathbf{y}=(y_1,\ldots,y_n)$ and $\mathbf{x}=(x_1,\ldots,x_n)$  the images of their Schwarz maps are $C$ and $X$, respectively, as in Theorem \ref{theoalg}. We will prove that if $d\in\mathbb{Z}_{>0}$ is the degree of $C$, then there exists a field extension of the field of definition of $L_{X}$, of degree at most $d$, over which $L_{X}(x)=0$ is gauge equivalent to a pullback of $L_C(y)=0$.

\begin{thm}\label{thetheo}
Let $G$ be a finite primitive algebraic subgroup of $GL_n(\mathbb{C})$ and let $C\subseteq\mathbb{P}^{n-1}(\mathbb{C})$ be an algebraic $G$-invariant curve of degree $d$ not contained in a projective hyperplane. Let $L_C(y)=0$ be an LODE such that for a fundamental system of solutions $(y_1,\ldots,y_n)$ the image of its Schwarz map is $C$, as in Theorem \ref{theoalg}. Let $K$ be the field of definition of $L_C$. 

Let $L(x)=0$ be an LODE such that for a fundamental system of solutions $(x_1,\ldots,x_n)$ the image of its Schwarz map is an algebraic $G$-invariant curve $X\subseteq\mathbb{P}^{n-1}(\mathbb{C})$ not contained in a projective hyperplane. Let $F_0$ be the field of definition of $L_{X}$. Then there exist a field extension $F_0\subseteq E$, of degree at most $d$, $n-2$ functions $f_1,\ldots,f_{n-2}\in E$, an abelian extension $E\subseteq E_1$, two functions $f,h\in E_1$, and $\Lambda\in\mathbb{Z}_{>0}$ such that
$$x_i+f_1\delta_0(x_i)+\ldots+f_{n-2}\delta_0^{n-2}(x_i)=f^{1/\Lambda}\cdot y_i\circ h,$$
where $\delta_0:F_0\rightarrow F_0$ is a non-trivial derivation. 
\end{thm}

\begin{proof}
We may assume that $y_1,\ldots,y_n$ are functions defined over an open set $U_0\subseteq C_0$, where $C_0$ is a Riemann surface such that $\mathbb{C}(C_0)=K$. Similarly, we may assume that $x_1,\ldots,x_n$ are functions defined over an open set $V_0\subseteq X_0$, where $X_0$ is a Riemann surface such that $\mathbb{C}(X_0)=F_0$.

Since $X$ is not contained in a proyective hyperplane, for a generic $t\in V_0$, we have that the vectors
\begin{align*}
\mathbf{x}(t) &=\left(x_1(t),\ldots,x_n(t)\right),\\
\delta_0\left(\mathbf{x}\right)(t) &=\left(\delta_0\left(x_1\right)(t),\ldots,\delta_0\left(x_n\right)(t)\right),\\
\vdots & \\
\delta_0^{n-2}\left(\mathbf{x}\right)(t) & =\left(\delta_0^{n-2}\left(x_1\right)(t),\ldots,\delta_0^{n-2}\left(x_n\right)(t)\right)
\end{align*}
are linearly independent in $\mathbb{C}^n$ and therefore, there exist $f_1(t),\ldots,f_{n-2}(t)\in \mathbb{C}$ such that
$$[\mathbf{x}+f_1\delta_0\left(\mathbf{x}\right)+\ldots+f_{n-2}\delta_0^{n-2}\left(\mathbf{x}\right)](t)\in C\subseteq\mathbb{P}^{n-1}(\mathbb{C}).$$

As in Section \ref{orbitspace}, let $\Lambda\in\mathbb{Z}_{>0}$ be such that the homogeneous elements of $\mathbb{C}[X_1,\ldots,X_n]^G$ of degree $\Lambda$ form a homogeneous coordinate system for $\mathbb{P}^{n-1}(\mathbb{C})/G$ and let $\Phi_1,\ldots,\Phi_M$ be a basis of $\mathbb{C}[X_1,\ldots,X_n]^G_\Lambda$. Let $\Psi_1,\ldots,\Psi_m\in\mathbb{C}[\Phi_1,\ldots,\Phi_M]\subseteq\mathbb{C}[X_1,\ldots,X_n]$ be homogeneous polynomials such that
$$C=\Big\{[X_1,\ldots,X_n]\in\mathbb{P}^{n-1}(\mathbb{C})\Big|\ \Psi_i([X_1,\ldots,X_n])=0,\ i=1,\ldots,m\Big\}.$$
Therefore, $f_1,\ldots,f_{n-2}$ satisfy the system of equations
$$\Psi_i\left(\mathbf{x}+f_1\delta_0\left(\mathbf{x}\right)+\ldots+f_{n-2}\delta_0^{n-2}\left(\mathbf{x}\right)\right)=0,\ i=1,\ldots,m.$$
Since $C$ is a curve of degree $d$, the system can be solved in a field extension $E$ of degree $d$ over the field of definition of the system. Now, we will prove that field of definition of the system is $F_0$. Let $d_i$ be the degree of $\Psi_i$. By Proposition \ref{propgau}, if $Q_1,\ldots, Q_{m_i}$ is a basis of $\mathbb{C}[Y_1,\ldots,Y_n]^G_{d_i}$, then
$$\Psi_i\left(\mathbf{x}+f_1\delta_0\left(\mathbf{x}\right)+\ldots+f_{n-2}\delta_0^{n-2}\left(\mathbf{x}\right)\right)=g_1Q_1(\mathbf{x})+\ldots+g_{m_i}Q_{m-i}(\mathbf{x}),$$
where $g_i\in \mathbb{C}[f_0,\ldots,f_{n-1}]^G_{d_i}$ with $f_0=1$ and $f_{n-1}=0$. By Galois correspondence, $Q_1(\mathbf{x}),\ldots,Q_{m-i}(\mathbf{x})\in K_1$ and therefore, the field of definition of the system is $K_1$.

Let $L'(x)=0$ be the LODE with fundamental system of solution $\mathbf{x}+f_1\delta_0\left(\mathbf{x}\right)+\ldots+f_{n-2}\delta_0^{n-2}\left(\mathbf{x}\right)$. Since the image of the Schwarz map of this fundamental system of solutions is $C$, the theorem follows from applying Theorem \ref{theopullback} to $L'(x)=0$.
\end{proof}

\begin{cor}\label{coro}
Under the hypothesis and notation of Theorem \ref{thetheo}, there exist $g_0,\ldots,g_n\in E$ such that
$$x_i=g_0f^{1/\Lambda}\cdot y_i\circ h+g_1\delta(f^{1/\Lambda}\cdot y_i\circ h)\ldots+g_n\delta^n(f^{1/\Lambda}\cdot y_i\circ h).$$
In particular, $L(x)=0$ is gauge equivalent over $E$ to an LODE projectively equivalent to a pullback by a map of $L_C(y)=0$ in $E_1$.
\end{cor}

\begin{proof}
Let
\[
L(x)=\delta^n(x)+a_{n-1}\delta^{n-1}(x)+\ldots+a_1\delta(x)+a_0x.
\]
By Theorem \ref{thetheo}, we have
$$x_i+f_1\delta_0(x_i)+\ldots+f_{n-2}\delta_0^{n-2}(x_i)=f^{1/\Lambda}\cdot y_i\circ h.$$
Differentiating with $\delta$ this equation $(n-1)$-times and using the relation
$$\delta^n(x)=-a_{n-1}\delta^{n-1}(x)-\ldots-a_1\delta(x)-a_0x,$$
we obtain
{\small $$
\left[\begin{array}{ccccc}
1 & f_{1} & \hdots &  f_{n-2} & 0\\
0 & 1+f'_{1} & \hdots & f_{n-3}+f'_{n-2} & f_{n-2}\\
\vdots & \vdots & \vdots & \vdots & \vdots\\
* & * & \hdots &  * & *
\end{array}\right]
\left[\begin{array}{c}
x_i\\
\delta(x_i)\\
\vdots\\
\delta^{(n-1)}(x_i)
\end{array}\right]=
\left[\begin{array}{c}
f^{1/\Lambda}\cdot y_i\circ h \\
\delta\big(f^{1/\Lambda}\cdot y_i\circ h\big) \\
\vdots\\
\delta^{n-1}\big(f^{1/\Lambda}\cdot y_i\circ h\big) 
\end{array}\right].
$$}
Since $x_i$, $i=1,\ldots,n$ and $f^{1/\Lambda}\cdot y_i\circ h$, $i=1,\ldots,n$ form systems of solutions to irreducible LODEs, then the vectors
$$\left\{\left[\begin{array}{c}
x_i\\
\delta(x_i)\\
\vdots\\
\delta^{(n-1)}(x_i)
\end{array}\right],\quad i=1,\ldots,n\right\}$$
and
$$\left\{\left[\begin{array}{c}
f^{1/\Lambda}\cdot y_i\circ h \\
\delta\big(f^{1/\Lambda}\cdot y_i\circ h\big) \\
\vdots\\
\delta^{n-1}\big(f^{1/\Lambda}\cdot y_i\circ h\big) 
\end{array}\right],\quad i=1,\ldots,n\right\}$$
are linearly independent and therefore, the matrix
$$\left[\begin{array}{ccccc}
1 & f_{1} & \hdots &  f_{n-2} & 0\\
0 & 1+f'_{1} & \hdots & f_{n-3}+f'_{n-2} & f_{n-2}\\
\vdots & \vdots & \vdots & \vdots & \vdots\\
* & * & \hdots &  * & *
\end{array}\right]$$
is invertible. The corollary follows from the first entry of the vector equality
{\small $$
\left[\begin{array}{c}
x_i\\
\delta(x_i)\\
\vdots\\
\delta^{(n-1)}(x_i)
\end{array}\right]=
\left[\begin{array}{ccccc}
1 & f_{1} & \hdots &  f_{n-2} & 0\\
0 & 1+f'_{1} & \hdots & f_{n-3}+f'_{n-2} & f_{n-2}\\
\vdots & \vdots & \vdots & \vdots & \vdots\\
* & * & \hdots &  * & *
\end{array}\right]^{-1}
\left[\begin{array}{c}
f^{1/\Lambda}\cdot y_i\circ h \\
\delta\big(f^{1/\Lambda}\cdot y_i\circ h\big) \\
\vdots\\
\delta^{n-1}\big(f^{1/\Lambda}\cdot y_i\circ h\big) 
\end{array}\right].
$$}
\end{proof}

\begin{rem}
Note that in the case $n=2$ we have
$$ \left[\begin{array}{cc}
1 & 0\\
* & *
\end{array}\right]\left[\begin{array}{c}
x_i\\ \delta(x_i)
\end{array}\right]=\left[\begin{array}{c}
f^{1/\Lambda}\cdot y_i\circ h \\
\delta\big(f^{1/\Lambda}\cdot y_i\circ h\big) \\
\end{array}\right].
$$
By multiplying both sides of the equation by the inverse of the matrix, we recover the closed form from Klein's theorem
$$x_i=f^{1/\Lambda}\cdot y_i\circ h.$$
\end{rem}

\section{Third order algebraic LODEs}

In this section we will present generalized hypergeometric equations that can be used as standard equations for third order algebraic LODEs. We will also give the formula to obtain the pullback function in Theorem \ref{thetheo} for these equations. Moreover, we will illustrate Corollary \ref{coro}  with an example.

There are, up to isomorphism, $8$ primitive subgroups of $SL_3(\mathbb{C})$. As in \cite{SINGER1993}, we define the following matrices

$$
E_1=\left[\begin{array}{ccc}
1 & 0 & 0\\
0 & \xi^4 & 0\\
0 & 0 & \xi
\end{array}\right],\quad 
E_2=\left[\begin{array}{ccc}
-1 & 0 & 0\\
0 & 0 & -1\\
0 & -1 & 0
\end{array}\right],\quad
E_3=\dfrac{1}{\sqrt{5}}\left[\begin{array}{ccc}
1 & 2 & 2\\
1 & s & t\\
1 & t & s
\end{array}\right],
$$
$$
E_4=\dfrac{1}{\sqrt{5}}\left[\begin{array}{ccc}
1 & 2\lambda_2 & 2\lambda_2\\
\lambda_1 & s & t\\
\lambda_1 & t & s
\end{array}\right], \quad
S=\left[\begin{array}{ccc}
\beta & 0 & 0\\
0 & \beta^2 & 0\\
0 & 0 & \beta^4
\end{array}\right], \quad
R=\dfrac{1}{\sqrt{7}i}\left[\begin{array}{ccc}
a  & b & c\\
b & c & a\\
c & a  & b
\end{array}\right],
$$
$$
S_1=\left[\begin{array}{ccc}
1 & 0 & 0\\
0 & \omega & 0\\
0 & 0 & \omega^2
\end{array}\right], \quad
T=\left[\begin{array}{ccc}
0 & 1 & 0\\
0 & 0 & 1\\
1 & 0  & 0
\end{array}\right], \quad
U=\left[\begin{array}{ccc}
\varepsilon & 0 & 0\\
0 & \varepsilon & 0\\
0 & 0  & \varepsilon\omega
\end{array}\right],
$$
$$
V=\rho\left[\begin{array}{ccc}
1 & 1 & 1\\
1 & \omega & \omega^2\\
1 & \omega^2  & \omega
\end{array}\right],\text{ and }
Z=\left[\begin{array}{ccc}
\omega & 0 & 0\\
0 & \omega & 0\\
0 & 0  & \omega
\end{array}\right].
$$
Here, $\xi$ is a primitive $5$th root of unity, $s=\xi^3+\xi^2$, and $t=\xi^4+\xi$. Note that $t-s=\sqrt{5}$. Furthermore, $\varepsilon$ is a primitive $9$th root of unity. Thus, $\varepsilon^6+\varepsilon^3+1=0$ and $\omega=-1-\varepsilon^3$ is a primitive $3$rd root of unity. We have that $\beta$ is a primitive $7$-th root of unity, $a=\beta^4-\beta^3$, $b=\beta^2-\beta^5$, $c=\beta-\beta^6$, $\lambda_1=\frac{-1+\sqrt{15}i}{4}$, $\lambda_2=\frac{-1-\sqrt{15}i}{4}$, and $\rho=\frac{1}{\omega-\omega^2}$. Note that $\beta^6+\beta^5+\beta^3-\beta^4-\beta^2-\beta=\sqrt{7}i$. Then, the primitive subgroups of $SL_3(\mathbb{C})$ are the following.

\begin{itemize}
\item The Klein group $G_{168}=\langle R,S,T\rangle$ and its direct product with the cyclic group of order three $G_{168}\times C_{3}=\langle R,S,T,Z\rangle$. Their projection into $PGL_3(\mathbb{C})$ are both isomorphic to $G_{168}$.
\item The group $H_{216}^{SL_3}=\langle S_1,T,V,U\rangle$.  Its projection into $PGL_3(\mathbb{C})$ is isomorphic to the Hessian group $H_{216}$.
\item The group $H_{72}^{SL_3}=\langle S_1,T,V,UVU^{-1}\rangle$. Its projection into $PGL_3(\mathbb{C})$ is the normal subgroup of the Hessian group $H_{72}$.
\item The group $F_{36}^{SL_3}=\langle S_1,T,V\rangle$. Its projection into $PGL_3(\mathbb{C})$ is the normal subgroup of the Hessian group $F_{36}$.
\item The Valentiner group $A_6^{SL_3}=\langle E_1,E_2,E_3,E_4\rangle$. Its projection into $PGL_3(\mathbb{C})$ is isomorphic to $A_6$.
\item The alternating group $A_5=\langle E_1,E_2,E_3\rangle$ and its direct product with the cyclic group of order three $A_5\times C_3=\langle E_1,E_2,E_3,Z\rangle$. Their projection into $PGL_3(\mathbb{C})$ are both isomorphic to $A_5$.
\end{itemize}

In order to simplify the computations involved in obtaining the coefficients $f_i$'s from Theorem \ref{thetheo}, for each maximal finite primitive group we will choose a hypergeometric standard equation that is projectively equivalent to one of the equations from Beuker and Heckman's list \cite{BEUKERS1989, KATO2006}. The equation for $F_{36}^{SL_3}$ is based on Geiselmann and Ulmer's equation \cite{GEISELMAN1997}.

\subsection{The Klein group $G_{168}$ and $G_{168}\times C_3$}

Since the groups $G_{168}$ and $G_{168}\times C_3$ have the same projection into $PGL_3(\mathbb{C})$, projective curves under one group are invariant under the other. Therefore, it suffices to produce standard equations for $G_{168}$.

The invariant subring $\mathbb{C}[X_1,X_2,X_3]^{G_{168}}$ is generated by
\begin{eqnarray*}
F_4 & = & X_1^3X_2+X_2^3X_3+X_3^3X_1,\\
F_6 & = & \frac{1}{54}\det\left[\begin{array}{ccc}
\partial^2 F_4/\partial X_1\partial X_1 & \partial^2 F_4/\partial X_1\partial X_2 & \partial^2 F_4/\partial X_1\partial X_3 \\
\partial^2 F_4/\partial X_2\partial X_1 & \partial^2 F_4/\partial X_2\partial X_2 & \partial^2 F_4/\partial X_2\partial X_3 \\
\partial^2 F_4/\partial X_3\partial X_1 & \partial^2 F_4/\partial X_3\partial X_2 & \partial^2 F_4/\partial X_3\partial X_3 \\
\end{array}\right],\\
F_{14} & =  & \frac{1}{9}\det\left[\begin{array}{cccc}
\partial^2 F_4/\partial X_1\partial X_1 & \partial^2 F_4/\partial X_1\partial X_2 & \partial^2 F_4/\partial X_1\partial X_3 & \partial F_6/\partial X_1\\
\partial^2 F_4/\partial X_2\partial X_1 & \partial^2 F_4/\partial X_2\partial X_2 & \partial^2 F_4/\partial X_2\partial X_3 & \partial F_6/\partial X_2\\
\partial^2 F_4/\partial X_3\partial X_1 & \partial^2 F_4/\partial X_3\partial X_2 & \partial^2 F_4/\partial X_3\partial X_3 & \partial F_6/\partial X_3\\
\partial F_6/\partial X_1 & \partial F_6/\partial X_2 & \partial F_6/\partial X_3 & 0
\end{array}\right],
\end{eqnarray*}
and
\begin{eqnarray*}
F_{21} & = & \frac{1}{14}\det\left[\begin{array}{ccc}
\partial F_4/\partial X_1 & \partial F_4/\partial X_2 & \partial F_4/\partial X_3 \\
\partial F_6/\partial X_1 & \partial F_6/\partial X_2 & \partial F_6/\partial X_3 \\
\partial F_{14}/\partial X_1 & \partial F_{14}/\partial X_2 & \partial F_{14}/\partial X_3
\end{array}\right].
\end{eqnarray*}
As a ring, $\mathbb{C}[X_1,X_2,X_3]^{G_{168}}$ is isomorphic to $\mathbb{C}[Z_4,Z_6,Z_{14},Z_{21}]/(T)$, where
\begin{eqnarray*}
T & = & Z_{21}^2+2048Z_4^9Z_6 -22016Z_4^6Z_6^3 +256Z_{14}Z_4^7 +60032Z_4^3Z_6^5 -1088Z_{14}Z_4^4Z_6^2\\ & & -1728Z_6^7 -1008Z_{14}Z_4Z_6^4 +88Z_{14}^2Z_4^2Z_6 -Z_{14}^3.
\end{eqnarray*}

The equation with solution ${}_3F_2(-1/42,5/42,17/42;1/3,2/3| t)$ is
\begin{eqnarray*}
0 & = & \left(\frac{d}{dt}\right)^3y+\frac{1}{2}\frac{7t-4}{t(t-1)}\left(\frac{d}{dt}\right)^2y+\frac{1}{252}\frac{387t-56}{t^2(t-1)}\left(\frac{d}{dt}\right)y \\
 & & \ -\frac{1}{74088}\frac{85}{t^2(t-1)}y .
\end{eqnarray*}
The image of the Schwarz map describes Klein's quartic, $F_4=0$. If $\mathbf{y}=(y_1,y_2,y_3)$ is a full system of solutions, we have
\begin{eqnarray*}
F_6(\mathbf{y}) & = & 1\\
\dfrac{F_{14}^3}{1728F_6^7}(\mathbf{y}) & = & t
\end{eqnarray*}
and therefore, the equation is standard.

\subsubsection{Example}

Here, we will ilustrate Corollary \ref{coro} using an equation with differential Galois group $G_{168}$. Let $\mathbf{x}=(x_1,x_2,x_3)$ be a full system of solutions to van der Put-Ulmer's equation \cite[Section 8.2.1, \emph{Branch type 2,7,7}]{VANDERPUT2000}
\begin{eqnarray*}
0 & = & \left(\frac{d}{dt}\right)^3x+\frac{7t-2}{t(t-1)}\left(\frac{d}{dt}\right)^2x +\frac{1}{28}\frac{288t^2-161t-7}{t^2(t-1)^2}\left(\frac{d}{dt}\right)x \\
 & & \ +\frac{1}{2744}\frac{6336t^3-5273t^2+343t-686}{t^3(t-1)^3}x.
\end{eqnarray*}
The image of the Schwarz map describes the curve $-\dfrac{7}{53}F_6^3-\dfrac{1}{8}F_4F_{14}+F_{4}^3F_6=0$ \cite{VANDERPUT2020}. For a generic $f_1$, we have
$$ F_4(\mathbf{x}+f_1\mathbf{x}') = \frac {p_1(t)p_2(t)p_3(t)}{38416 \left( t-1
 \right) ^{7}{t}^{6}},$$
where $p_1(t) = 14\,{t}^{2}-14\,t-(19\,t-7)\,f_{{1}}$, 
$p_2(t)=14\,{t}^{2}-14\,t-(16\,t-7)\,f_{{1}}$, and
$p_3(t)=1372\,{t}^{4}-2744\,{t}^{3}+1372\,{t}^{2}-(3430\,{t}^{3}f_{{1}}-4802\,{t}^{
2}+1372\,t)f_{{1}}+(2146\,{t}^{2}-1715\,t+343)\,{f_{{1}}}^{2}$.
Let $f_1(t)=(14\,{t}^{2}-14\,t)/(19\,x-7)$. Then, $p_1(t)=0$ and the map $[\mathbf{x}+f_1\mathbf{x}']$ describes Klein's quartic. Let
$$ f(t) = F_6(\mathbf{x}+f_1\mathbf{x}')=\dfrac{3^6t^3}{(t-1)^4(19t-7)^6}$$
and
$$h(t) = \dfrac{F_{14}^3}{1728F_6^7}(\mathbf{x}+f_1\mathbf{x}')=\dfrac{(t+3)^3}{27(t-1)^2}.$$
From Theorem \ref{thetheo}, we obtain $\mathbf{x}(t)+f_1(t)\mathbf{x}'(t)=f(t)^{1/6}\mathbf{y}\left(h(t)\right)$. If we denote
$$F_{G_{168}}(t)={}_3F_2(-1/42,5/42,17/42;1/3,2/3| t),$$
then, by Corollary \ref{coro}, we obtain a solution
\begin{align*}
x(t) & =\dfrac{1932781}{6049137024}\dfrac{1}{t^{1/2}(t-1)^{17/3}}\Bigg\{ t(t-9)^2(t+3)^4F''_{G_{168}}\left(\dfrac{(t+3)^3}{27(t-1)^2}\right)\\
 & \quad +\dfrac{5821200}{113693}(t-1)^2\bigg\{(t+3)(t^3-\dfrac{351}{55}t^2+\dfrac{243}{5}t+\dfrac{567}{55})F'_{G_{168}}\left(\dfrac{(t+3)^3}{27(t-1)^2}\right)\\
 & \qquad +\dfrac{433944}{4675} (t-\dfrac{21}{41})(t-1)^2F_{G_{168}}\left(\dfrac{(t+3)^3}{27(t-1)^2}\right)\bigg\}\Bigg\}.
\end{align*}

\subsection{The Hessian group $H_{216}$ and its subgroups $H_{72}$ and $F_{36}$}

We define $P =X_1X_2X_3$, $S = X_1^3+X_2^3+X_3^3$, $Q=X_1^3X_2^3+X_1^3X_3^3+X_2^3X_3^3$, $R = (X_1^3-X_2^3)(X_1^3-X_3^3)(X_2^3-X_3^3)$,
$F_6 = S^2-12Q$, $\Phi_6 =S^2-18P^2-6PS$, $F_{12} =S^4+216P^3S$, $\Phi_{12} =P(27P^3-S^3)$, and $\Psi_{12}=PS^3+3P^2S^2-18P^3S.$

Then, the ring of invariants for $F_{36}^{SL_3}$, $H_{72}^{SL_3}$,  $H_{216}^{SL_3}$ are
$$\mathbb{C}[X_1,X_2,X_3]^{F_{36}^{SL_3}}  =\mathbb{C}[F_6,\Phi_6,R,F_{12},\Psi_{12}],$$
$$\mathbb{C}[X_1,X_2,X_3]^{H_{72}^{SL_3}}  =\mathbb{C}[F_6,R,F_{12},\Phi_6^2],$$
and
$$\mathbb{C}[X_1,X_2,X_3]^{H_{216}^{SL_3}}  =\mathbb{C}[R,\Phi_{12},F_6F_{12},F_6^3],$$
respectively. As rings, they are isomorphic to
$$\mathbb{C}[Z_6,Y_6,Z_9,Z_{12},Y_{12}]/(T_{18},T_{24}),$$
$$\mathbb{C}[Z_6,Z_9,Z_{12},X_{12}]/(T_{36}),$$
and
$$\mathbb{C}[Z_9,Y_{12},Z_{18},Y_{18}]/(T_{54}),$$
respectively, where
$$ T_{18} =432Z_9^2-Z_6^3+3Z_6Z_{12}-2Y_6^3-36Y_6Y_{12},$$
$$T_{24}  = Z_{12}^2-Y_6(Z_{12}+12Y_{12})+12Y_{12}^2,$$
$$T_{36}  =(432Z_9^2+3Z_6Z_{12}-Z_6^3)^2-4(X_{12}^3-3Z_{12}X_{12}^2+3Z_{12}^2X_{12}),$$
and
$$T_{54}  =Z_{18}^3-\dfrac{1}{4}\big((432Z_9^2-Y_{18}+3Z_{18})^2-4\cdot 1728Y_{12}^3\big)Y_{18}.$$

The equation with solution ${}_3F_2(17/36, 2/9, -1/36;1/3, 2/3| \frac{1}{t})$ is
\begin{eqnarray*}
0 & = & \left(\frac{d}{dt}\right)^3y+\frac{1}{3}\frac{11t-6}{t(t-1)}\left(\frac{d}{dt}\right)^2y+\frac{1}{432}\frac{757t-96}{t^2(t-1)}\left(\frac{d}{dt}\right)y \\
 & & \ -\frac{1}{5832}\frac{17}{t^3(t-1)}y .
\end{eqnarray*}
The image of the Schwarz map describes the curve $F_6=0$. If $\mathbf{y}=(y_1,y_2,y_3)$ is a full system of solutions, then
\begin{eqnarray*}
R(\mathbf{y}) & = & 1\\
\dfrac{6^6R^4}{\Phi_{12}^3}(\mathbf{y}) & = & t
\end{eqnarray*}
and therefore, the equation is standard.

Note that since
$$\Phi_{12}=\dfrac{\Phi_6^2-F_{12}}{12},$$
we can obtain the value of $\Phi_{12}(\mathbf{y})$ from the invariants under the actions of $H_{72}^{SL_3}$ or $F_{36}^{SL_3}$.

\begin{rem} There is no hypergeometric equation with Galois group $H_{72}^{SL_3}$. On the other hand, there exist hypergeometric equations with Galois group $F_{36}^{SL_3}$. For example, the equation with solution ${}_3F_2(-1/12,1/6,5/12;1/4,3/4| \frac{1}{t})$
\begin{eqnarray*}
0 & = & \left(\frac{d}{dt}\right)^3y+\frac{1}{2}\frac{8t-5}{t(t-1)}\left(\frac{d}{dt}\right)^2y+\frac{5}{48}\frac{21t-5}{t^2(t-1)}\left(\frac{d}{dt}\right)y \\
 & & \ -\frac{1}{864}\frac{5}{t^3(t-1)}y .
\end{eqnarray*}
Let $F_3=6\sqrt{3}P+(\sqrt{3}+3)S$ and $\Phi_3=6\sqrt{3}P+(\sqrt{3}-3)S$. The polynomials $F_3$ and $\Phi_3$ are semi-invariants of degree $3$ such that $F_3\Phi_3=\Phi_6$. The image of the Schwarz map describes the elliptic curve $F_3=0$. If $\mathbf{y}=(y_1,y_2,y_3)$ is a full system of solutions, then
\begin{eqnarray*}
F_6(\mathbf{y}) & = & 1\\
\dfrac{F_{6}^3}{F_{6}^3-432R^2}(\mathbf{y}) & = & t
\end{eqnarray*}
and therefore, the equation is standard.
\end{rem}

\subsection{The alternating group $A_6$}

The invariant subring $\mathbb{C}[X_1,X_2,X_3]^{A_{6}^{SL_3}}$ is generated by
\begin{eqnarray*}
F_6 & = & 10X_1^3X_2^3+9X_1^5X_3+9X_2^5X_3-45X_1^2X_2^2X_3^2-135X_1X_2X_3^4+27X_3^6,\\
F_{12} & = & -\frac{1}{20250}\det\left[\begin{array}{ccc}
\partial^2 F_6/\partial X_1\partial X_1 & \partial^2 F_6/\partial X_1\partial X_2 & \partial^2 F_6/\partial X_1\partial X_3 \\
\partial^2 F_6/\partial X_2\partial X_1 & \partial^2 F_6/\partial X_2\partial X_2 & \partial^2 F_6/\partial X_2\partial X_3 \\
\partial^2 F_6/\partial X_3\partial X_1 & \partial^2 F_6/\partial X_3\partial X_2 & \partial^2 F_6/\partial X_3\partial X_3 \\
\end{array}\right],\\
F_{30} & =  & \frac{1}{24300}\det\left[\begin{array}{cccc}
\partial^2 F_6/\partial X_1\partial X_1 & \partial^2 F_6/\partial X_1\partial X_2 & \partial^2 F_6/\partial X_1\partial X_3 & \partial F_{12}/\partial X_1\\
\partial^2 F_6/\partial X_2\partial X_1 & \partial^2 F_6/\partial X_2\partial X_2 & \partial^2 F_6/\partial X_2\partial X_3 & \partial F_{12}/\partial X_2\\
\partial^2 F_6/\partial X_3\partial X_1 & \partial^2 F_6/\partial X_3\partial X_2 & \partial^2 F_6/\partial X_3\partial X_3 & \partial F_{12}/\partial X_3\\
\partial F_{12}/\partial X_1 & \partial F_{12}/\partial X_2 & \partial F_{12}/\partial X_3 & 0
\end{array}\right],
\end{eqnarray*}
and
\begin{eqnarray*}
F_{45} & = & \frac{1}{4860}\det\left[\begin{array}{ccc}
\partial F_6/\partial X_1 & \partial F_6/\partial X_2 & \partial F_6/\partial X_3 \\
\partial F_12/\partial X_1 & \partial F_12/\partial X_2 & \partial F_12/\partial X_3 \\
\partial F_{30}/\partial X_1 & \partial F_{30}/\partial X_2 & \partial F_{30}/\partial X_3
\end{array}\right].
\end{eqnarray*}
As a ring, $\mathbb{C}[X_1,X_2,X_3]^{A_{6}^{SL_3}}$ is isomorphic to $\mathbb{C}[Z_6,Z_{12},Z_{30},Z_{45}]/(T)$, where
\begin{eqnarray*}
T & = & 4Z_{6}^13Z_{12}+80Z_{6}^11Z_{12}^2+816Z_{6}^9Z_{12}^3+18Z_{6}^10Z_{30}+4376Z_{6}^7Z_{12}^4+198Z_{6}^8Z_{12}Z_{30}\\
& & \quad +13084Z_{6}^5Z_{12}^5+954Z_{6}^6Z_{12}^2Z_{30}+12312Z_{6}^3Z_{12}^6-198Z_{6}^4Z_{12}^3Z_{30}+5616Z_{6}Z_{12}^7\\
& & \quad\ -162Z_{6}^5Z_{30}^2-5508Z_{6}^2Z_{12}^4Z_{30}-1944Z_{6}^3Z_{12}Z_{30}^2-1944Z_{12}^5Z_{30}-1458Z_{6}Z_{12}^2Z_{30}^2\\
& & \ \quad\ +729Z_{30}^3-19683Z_{45}^2.\\
\end{eqnarray*}

The equation with solution ${}_3F_2(-1/60,11/60,7/12;1/2,3/4| t)$ is
\begin{eqnarray*}
0 & = & \left(\frac{d}{dt}\right)^3y+\frac{3}{4}\frac{5t-3}{t(t-1)}\left(\frac{d}{dt}\right)^2y+\frac{1}{1200}\frac{2213t-450}{t^2(t-1)}\left(\frac{d}{dt}\right)y \\
 & & \ -\frac{1}{43200}\frac{77}{t^2(t-1)}y .
\end{eqnarray*}
The image of the Schwarz map the curve $F_6=0$. If $\mathbf{y}=(y_1,y_2,y_3)$ is a full system of solutions, we have
\begin{eqnarray*}
F_{12}(\mathbf{y}) & = & 1\\
\dfrac{3F_{30}^2}{8F_{12}^5}(\mathbf{y}) & = & t
\end{eqnarray*}
and therefore, the equation is standard.

\subsection{The alternating group $A_5$ and the group $A_5\times C_3$}

The invariant subring $\mathbb{C}[X_1,X_2,X_3]^{A_{6}^{SL_3}}$ is generated by
\begin{eqnarray*}
F_2 & = & X_1^2+X_2X_3,\\
F_{6} & = & 8X_1^4X_2X_3-2X_1^2X_2^2X_3^2-X_1(X_2^5+X_3^5)+X_2^3X_3^3,\\
F_{10} & =  & 320 X_1^6 X_2^2 X_3^2-160 X_1^4 X_2^3 X_3^3+20 X_1^2 X_2^4 X_3^4+6 X_2^5 X_3^5\\
 & & \quad-4 X_1 (X_2^5+X_3^5) (32 X_1^4-20 X_1^2 X_2 X_3+5 X_2^2 X_3^2)+X_2^{10}+X_3^{10},
\end{eqnarray*}
and
\begin{eqnarray*}
F_{15} & = & (X_2^5-X_3^5) (-1024 X_1^{10}+3840 X_1^8 X_2 X_3-3840 X_1^6 X_2^2 X_3^2+1200 X_1^4 X_2^3 X_3^3\\
& & \quad -100 X_1^2 X_2^4 X_3^4+X_2^{10}+X_3^{10}+2 X_2^5 X_3^5\\
& & \qquad +X_1 (X_2^5+X_3^5) (352 X_1^4-160 X_1^2 X_2 X_3+10 X_2^2 X_3^2)).
\end{eqnarray*}
As a ring, $\mathbb{C}[X_1,X_2,X_3]^{A_{5}}$ is isomorphic to $\mathbb{C}[Z_2,Z_{6},Z_{10},Z_{15}]/(T)$, where
\begin{eqnarray*}
T & = &  Z_{15}^2+1728 Z_{6}^5- Z_{10}^3-720 Z_{2} Z_{6}^3 Z_{10}+80 Z_{2}^2 Z_{6} Z_{10}^2-64 Z_{2}^3(5 Z_{6}^2- Z_{2} Z_{10})^2.\\
\end{eqnarray*}

The equation with solution ${}_3F_2(-1/30,1/6,11/30;1/3,2/3| t)$ is
\begin{eqnarray*}
0 & = & \left(\frac{d}{dt}\right)^3y+\frac{1}{2}\frac{7t-4}{t(t-1)}\left(\frac{d}{dt}\right)^2y+\frac{1}{900}\frac{1389t-200}{t^2(t-1)}\left(\frac{d}{dt}\right)y \\
 & & \ -\frac{1}{5400}\frac{11}{t^2(t-1)}y .
\end{eqnarray*}
It is the symmetric square of the equation with solution ${}_2F_1(-1/60,11/60;2/3| t)$, a standard equation for $A_5^{SL_2}$. In particular the image of the Schwarz map is the rational curve $F_2=0$. If $\mathbf{y}=(y_1,y_2,y_3)$ is a full system of solutions, we have
\begin{eqnarray*}
F_{6}(\mathbf{y}) & = & 1\\
\dfrac{F_{10}^3}{1728F_{6}^5}(\mathbf{y}) & = & t
\end{eqnarray*}
and therefore, the equation is standard.

\bibliographystyle{plain}
\bibliography{DGT}

\end{document}